\documentclass[12pt]{article}
\usepackage{amsmath,amssymb}
\usepackage{tikz}
\usepackage{graphicx}
\usepackage{fancyhdr}
\usepackage{enumitem}
\usepackage{bbm}
\usepackage{graphicx}
\usepackage{chngpage}
\usepackage{calc}
\usepackage{bigints}
\usepackage{tikz}
\usepackage{array}
\usepackage{booktabs}
\usepackage{rotating}
\usepackage{multirow}
\usepackage{adjustbox}
\usepackage{tabularx}
\usepackage{verbatim}
\usepackage{mathtools}
\usepackage{ragged2e}
\usepackage[makeroom]{cancel}
\usepackage{enumitem}

\newtheorem{theorem}{Theorem}[section]

\newtheorem{remark}[theorem]{Remark}
\newtheorem{assumption}[theorem]{Assumption}

\newenvironment{proof}{\smallskip\par{\sc Proof.}\enspace}%
 {{\unskip\nobreak\hfil\penalty50\hskip2em
          \hbox{}\nobreak\hfil{\rule[-1pt]{5pt}{10pt}}
          \parfillskip=0pt\finalhyphendemerits=0
          \par\medskip}} 

\makeatletter
\def\section{\@startsection {section}{1}{\z@}{3.25ex plus 1ex minus
 .2ex}{1.5ex plus .2ex}{\large\bf}}
\def\subsection{\@startsection{subsection}{2}{\z@}{3.25ex plus 1ex minus
 .2ex}{1.5ex plus .2ex}{\normalsize\bf}}
\@addtoreset{equation}{section} 
\makeatother

\chardef\bslash=`\\ 

\def\N{\mathbb{N}}
\def\P{\mathbb{P}}
\def\E{\mathbb{E}}

\newcommand{\eval}[2][\right]{\relax
\ifx#1\right\relax \left.\fi#2#1\rvert}
 
\numberwithin{equation}{section}
\newcount\madechanges
\def\caution#1{\ifnum \madechanges=1 \affixmessage{#1}%
\else \relax \fi}
\def\affixmessage#1{\marginpar{{\footnotesize  \em #1} \openup
    -.3\baselineskip }}
\title{ On a class of stochastic differential equations with random and H\"older continuous coefficients arising in biological modeling }

\author{Enrico Bernardi\thanks{Dipartimento di Scienze Statistiche Paolo Fortunati, Università di Bologna,
Bologna, Italy. \textbf{e-mail}: enrico.bernardi@unibo.it} \and Vinayak Chuni\thanks{Dipartimento di Scienze Statistiche Paolo Fortunati, Università di Bologna,
Bologna, Italy. \textbf{e-mail}: vinayak.chuni2@unibo.it} \and  Alberto Lanconelli\thanks{Dipartimento di Scienze Statistiche Paolo Fortunati, Università di Bologna,
Bologna, Italy. \textbf{e-mail}: alberto.lanconelli2@unibo.it}}

\date{\today}

\begin{document}

\maketitle

\numberwithin{equation}{section}

\bigskip

\begin{abstract}
Inspired by the paper Greenhalgh et al. \cite{Mao paper} we investigate a class of two dimensional stochastic differential equations related to susceptible-infected-susceptible epidemic models with demographic stochasticity. While preserving the key features of the model considered in \cite{Mao paper}, where an \emph{ad hoc} approach has been utilized to prove existence, uniqueness and non explosivity of the solution, we consider an encompassing family of models described by a stochastic differential equation with random and H\"older continuous coefficients. We prove the existence of a unique strong  solution by means of a Cauchy-Euler-Peano approximation scheme which is shown to converge in the proper topologies to the unique solution.
 \end{abstract}



\allowdisplaybreaks

\section{Introduction}

Susceptible-infected-susceptible (SIS) epidemic model is one of the most popular models for how diseases spread in a population. In such a model an individual starts off being susceptible to a disease and at some point of time gets infected and then recovers after some time  becoming susceptible again. The literature of such mathematical models is very rich: for probabilistic/stochastic models one may look for instance at Allen \cite{Allen paper}, Allen and Burgin \cite{AB}, A. Gray et al. \cite{GGHMP}, Hethcote and van den Driessche \cite{HD}, Kryscio and  Lefévre \cite{KL}, McCormack and Allen \cite{MA} and Nasell \cite{N}. We also refer the reader to the detailed account presented in Greenhalgh et al. \cite{Mao paper} for an overview on both deterministic and stochastic models.\\ 
The focus of the present paper is on the model presented in \cite{Mao paper}. One of its distinguishing features is the nature of the births and deaths that are regarded as stochastic processes with per capita disease contact rate depending on the population size. Contrary to many other previously proposed models, this stochasticity produces a variable population size which turns out to be a reasonable assumption for slowly spreading diseases. \\
From a mathematical point of view, the SIS model proposed in \cite{Mao paper} amounts at the following two dimensional stochastic differential equation for the vector $(S_t, I_t)$ where $S_t$ and $I_t$ stand for the number of susceptible and infected individuals at time $t$, respectively:
\begin{eqnarray}
\left\{ \begin{array}{ll}
dS&=\left[  -\frac{\lambda(N)SI}{N}+(\mu+\gamma)I \right]dt+\sqrt{\frac{\lambda(N)SI}{N}+(\mu+\gamma)I+2\mu S}dW_3 \label{Mao 2 intro} \\
dI&=\left[  \frac{\lambda(N)SI}{N}-(\mu+\gamma)I \right]dt+\sqrt{\frac{\lambda(N)SI}{N}+(\mu+\gamma)I}dW_4.
\end{array}\right.
\end{eqnarray}
Here, $N:=S+I$ denotes the total population size while $\mu$, $\gamma$ and $\lambda:[0,+\infty[\to [0,+\infty[$ are suitably chosen parameters. The system (\ref{Mao 2 intro}) is driven by the two dimensional correlated Brownian motion $(W_3,W_4)$ resulting from a certain application of the martingale representation theorem (see Section \ref{section Mao} below for technical details). The system (\ref{Mao 2 intro}) is then shown to be equivalent to the triangular system
\begin{eqnarray}
\left\{ \begin{array}{ll}\label{Mao triangular SDE intro}
dI&=\left[  \frac{\lambda(N)}{N}(N-I)I-(\mu+\gamma)I \right]dt+\sqrt{\frac{\lambda(N)}{N}(N-I)I+(\mu+\gamma)I}dW_4\\
dN&=\sqrt{2\mu N}dW_5
\end{array}\right.
\end{eqnarray}
where now the second equation, the so-called square root process (see  for instance the book by Mao \cite{Mao book} for the properties of this process), is independent of the first one. To prove the existence of a solution to the first equation in (\ref{Mao triangular SDE intro}) the authors resort to Theorem 2.2 in Chapter IV of Ikeda and Watanabe \cite{Ikeda Watanabe} while for the uniqueness they need to construct a localized version of Theorem 3.2, Chapter IV in \cite{Ikeda Watanabe}. The equation for $I$ in (\ref{Mao triangular SDE intro}) exhibits random (for the dependence on the process $N$) and H\"older continuous (for the presence of the square root in the diffusion term) coefficients resulting in a stochastic differential equation for which the issue of the existence of a unique solution has not been addressed in the literature yet.\\
Our aim in the present paper is to propose a more general approach allowing for the investigation of a richer family of models characterized by the same distinguishing features of the model analyzed in \cite{Mao paper}.\\
The paper is articulated as follows: In Section 2 we present a general review using the exposition in the book by Allen (see \cite{Allen book}) of a two-state dynamics leading to a Fokker-Planck partial differential equation and its associated stochastic system. This is followed by Section 2.1 where we consider the more specific situation of a bio-demographic model like the one presented in \cite{Mao paper}. Our idea is to embed the rather special system of SDE's of the model in a slightly more encompassing class, like the one in (\ref{SDE}) below, in order to establish a general proof of strong existence and uniqueness. Our technique relies on the construction of an explicit approximating sequence of stochastic processes (inspired by the work of Zubchenko \cite{Zubchenko}) in such a way that all the relevant features of the solution appear to be directly constructed from scratch. In Section 3 we give a detailed proof of existence and uniqueness of the SDE (\ref{SDE}). We would like to point out that systems of SDE's with non-Lipschitz or H\"older coefficients exhibit non-standard difficulties as far as general results for existence and uniqueness are concerned. This model conforms to the aforementioned difficulties and that is what has motivated us in approaching  the problem. Our idea has been to how we could encase the model proposed in \cite{Mao paper} within a more general framework , thus bypassing some of the computations done there, and hopefully allowing for larger class of models to be treated. 

\section{A general two-state system}

In this section we review the construction of a general two-state system presented in the book by Allen (\cite{Allen book}). The model will then be made concrete through the assumptions contained in the paper by Greenhalgh et al. (\cite{Mao paper}) and this will lead to the class of stochastic differential equations investigated in the present manuscript. 

\begin{figure}[h]
\begin{center}
\begin{tikzpicture}[very thick]

\shade[draw=black!50,fill=black!20,thick] (0,0) rectangle +(3,3);
\shade[draw=black!50,fill=black!20,thick] (8,0) rectangle +(3,3);

\node (local1) at ( 1.5,1.5) {$S_{1}(t)$};
\node (local2) at (9.5,1.5) {$S_{2}(t)$};
\node (local3) at ( 3,1.5) {};
\node (local4) at ( 8,1.5) {};

\node (local5) at ( 3,1.1) {};
\node (local6) at ( 8,1.1) {};

\node (local7) at ( 0.1,0) {};
\node (local8) at ( 0.1,-2) {};

\node (local9) at ( 1.5,0) {};
\node (local10) at ( 1.5,-2) {};

\node (local11) at ( 2.7,0) {};
\node (local12) at ( 2.7,-1.8) {};

\node (local13) at ( 8.3,0) {};
\node (local14) at ( 8.3,-1.8) {};

\node (local15) at ( 9.5,0) {};
\node (local16) at ( 9.5,-1.8) {};

\node (local17) at ( 10.9,0) {};
\node (local18) at ( 10.9,-1.8) {};

\node (local19) at ( 4.5,-1.55) {};
\node (local20) at ( 4.5,-3) {};

\node (local21) at ( 6.5,-1.55) {};
\node (local22) at ( 6.5,-3) {};

\draw [->] (local3.east) -- node [above] {$5$}  (local4.west);
\draw [->] (local6.west) -- node [below] {$6$}  (local5.east);

\draw [->] (local7.south) -- node [left] {$1$}  (local8.north);
\draw [->] (local10.north) -- node [right] {$2$}  (local9.south);

\draw [-] (local12.north) -- node [] {}  (local11.south);


\draw [-] (local11.south) -- node [] {}  (local12.north) -- node []
{} (local14.north) -- node [] {}  (local13.south) ;

\draw [->] (local15.south) -- node [left] {$3$}  (local16.north);

\draw [->] (local18.north) -- node [right] {$4$}  (local17.south);

\draw [->] (local19.south) -- node [left] {$7$}  (local20.north);

\draw [<-] (local21.south) -- node [right] {$8$}  (local22.north);

\end{tikzpicture}

\end{center}
\caption{A two-state dynamical process}
\label{fig:state}
\end{figure}
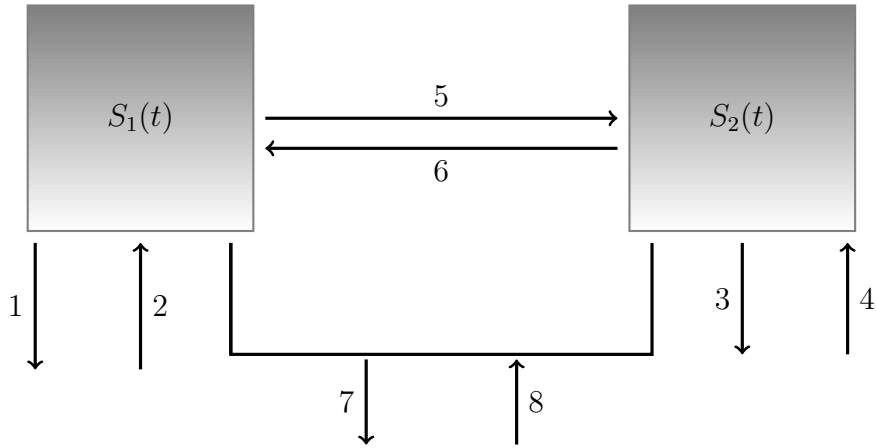

We begin by considering a representative two-state dynamical process which is illustrated in Figure \ref{fig:state}. Let $S_1(t)$ and $S_2(t)$ represent the values of the two states of the system at time $t$. It is assumed that in a small time interval $\Delta t$, state $S_1$ can change by $- \lambda_1$, $0$ or $\lambda_1$ and state $S_2$ can change by $- \lambda_2$, $0$ or $\lambda_2$, where $\lambda_1,\lambda_2 \ge 0$. Let $\Delta S := [\Delta S_1,\Delta S_2]^T$ be the change in a small time interval $\Delta t$.  As illustrated in Figure \ref{fig:state} , there are eight possible changes for the two states in the time interval $\Delta t$ not including the case where there is no change in the time interval. The possible changes and the probabilities of these changes are given in Table \ref{tab:table1}. It is assumed that the probabilities are given to $O((\Delta t)^2)$. For example, change $1$ represents a loss of $\lambda_1$ in $S_1$ with probability $d_1 \Delta t$, change $5$ represents a transfer of $\lambda_1$ out of state $S_1$ with a corresponding transfer of $\lambda_2$ into state $S_2$ with probability $m_{12} \Delta t$ and change $7$ represents a simultaneous reduction in both states $S_1$ and $S_2$. As indicated in the table, all probabilities may depend on $S_1(t)$, $S_2(t)$ and the time $t$. Also notice that it is assumed that the probabilities for the changes are proportional to the time interval $\Delta t$.
\begin{table}[h!]
  \centering
     \caption{Possible changes in the representative two-state system
with the corresponding probabilities}
     \label{tab:table1}
     \begin{tabular}{lcl}
       \toprule Change & & Probability\\
       \midrule $
\Delta \mathbf{S}^{(1)} = [-1,0]^{T} $&  & $ p_{1} =
d_{1}(t,S_{1},S_{2})\Delta t $\\ $ \Delta \mathbf{S}^{(2)} = [1,0]^{T}
       $ & & $ p_{2} = b_{1}(t,S_{1},S_{2})\Delta t $\\
       $ \Delta
\mathbf{S}^{(3)} = [0,-1]^{T} $ &  & $ p_{3} =
                                     d_{2}(t,S_{1},S_{2})\Delta t $\\
       $ \Delta \mathbf{S}^{(4)} = [0,1]^{T}
       $&  & $ p_{4} = b_{2}(t,S_{1},S_{2})\Delta t $\\
       $ \Delta
\mathbf{S}^{(5)} = [-1,1]^{T} $ & & $ p_{5} =
                                    m_{12}(t,S_{1},S_{2})\Delta t $\\
       $ \Delta \mathbf{S}^{(6)} =
       [1,-1]^{T} $&  & $ p_{6} = m_{21}(t,S_{1},S_{2})\Delta t $\\
       $ \Delta
\mathbf{S}^{(7)} = [-1,-1]^{T} $ & & $ p_{7} =
                                    m_{11}(t,S_{1},S_{2})\Delta t $\\
       $ \Delta \mathbf{S}^{(8)} =
       [1,1]^{T} $ & & $ p_{8} = m_{22}(t,S_{1},S_{2})\Delta t $\\
       $ \Delta
\mathbf{S}^{(9)} = [0,0]^{T} $ & & $ p_{9} = 1 - \sum_{j=1}^{8}p_{j}
                                   $\\
       \bottomrule
     \end{tabular}
   \end{table}

\noindent It is useful to calculate the mean vector and covariance matrix for the change $\Delta S = [\Delta S_1, \Delta S_2]^T$ fixing the value of $S$ at time $t$. Using the table below,
 \begin{eqnarray*}
E[\Delta S] = \sum_{j=1}^9 p_j \Delta S^{(j)}= \begin{bmatrix} (-d_{1}
+b_{1} -m_{12} +m_{21} +m_{22} -m_{11}) \lambda_1  \\ (-d_{2} +b_{2} +m_{12}
-m_{21} +m_{22} -m_{11}) \lambda_2
   \end{bmatrix} \Delta t
   \end{eqnarray*}
 \begin{eqnarray*}
   E[\Delta S (\Delta S)^T] &=& \sum_{j=1}^9 p_j (\Delta S^{(j)}) (\Delta S^{(j)})^T \\ 
&=&\begin{bmatrix} (d_{1}
+b_{1} +m_a) \lambda_1^2 & (-m_{12}-m_{21}+m_{22}+m_{11})\lambda_1 \lambda_2   \\ (-m_{12}-m_{21}+m_{22}+m_{11})\lambda_1 \lambda_2 & (d_{2}
+b_{2} +m_a) \lambda_2^2
   \end{bmatrix} \Delta t
   \end{eqnarray*}
   where we set $m_a:=m_{12}+m_{21}+m_{11}+m_{22}$. Notice that the covariance matrix is set equal to $E(\Delta S(\Delta S)^T )/\Delta t$ because $E(\Delta S)(E(\Delta S))^T = O((\Delta t)^2)$.  We now define 
  \begin{eqnarray}\label{mean and variance}
   \mu(t,S_1,S_2)=E[\Delta S] /\Delta t \quad\text{ and  }\quad V(t,S_1,S_2)= E[\Delta S (\Delta S)^T]/ \Delta t
   \end{eqnarray}
and we denote by $B(t, S_1, S_2)$ the symmetric square root matrix of $V$. A forward Kolmogorov equation can be determined for the probability distribution at time $t +\Delta t$ in terms of the distribution at time $t$. If we write $p(t, x_1, x_2)$ for the probability that $S_1(t) = x_1$ and $S_2(t) = x_2$, then referring to Table \ref{tab:table1} we get
\begin{equation}\label{Equation5.1Allen}
p(t + \Delta t, x_1, x_2) = p(t, x_1, x_2) + \Delta t  \sum_{i=1}^{10} T_i 
\end{equation}
where 
\begin{eqnarray*}
T_1 &=&p(t,x_1,x_2)(-d_1(t,x_1,x_2)-b_1(t,x_1,x_2)-d_2(t,x_1,x_2)-b_2(t,x_1,x_2))\\
T_2 &=&p(t,x_1,x_2)(-m_a(t,x_1,x_2))\\
T_3 &=&p(t,x_1 +\lambda_1,x_2)d_1(t,x_1 +\lambda_1,x_2)\\
T_4 &=&p(t,x_1 -\lambda_1,x_2)b_1(t,x_1 -\lambda_1,x_2)\\
T_5 &=&p(t,x_1,x_2 -\lambda_2)b_2(t,x_1,x_2 -\lambda_2)\\
T_6 &=&p(t,x_1,x_2 +\lambda_2)d_2(t,x_1,x_2 +\lambda_2)\\
T_7 &=&p(t,x_1 +\lambda_1,x_2 -\lambda_2)m_{12}(t,x_1 +\lambda_1,x_2 -\lambda_2)\\
T_8 &=&p(t,x_1 -\lambda_1,x_2 +\lambda_2)m_{21}(t,x_1 -\lambda_1,x_2 +\lambda_2)\\
T_9 &=&p(t,x_1 +\lambda_1,x_2 +\lambda_2)m_{11}(t,x_1 +\lambda_1,x_2 +\lambda_2)\\
T_{10}&=&p(t,x_1 -\lambda_1,x_2 -\lambda_2)m_{22}(t,x_1 -\lambda_1,x_2 -\lambda_2).
\end{eqnarray*}
Now, expanding out the terms $T_3$  through $T_{10}$ in second order Taylor polynomials around the point $(t, x_1, x_2)$, it follows that
\begin{eqnarray*}
T_3 &\approx&  pd_1 +\partial_{x_1}( pd_1 )\lambda_1 + 1/2 \partial_{x_1x_2}^2(pd_1) \lambda_1^2\\
T_4 &\approx&  pb_1 -\frac { \partial( pb_1 )}{\partial x_1} \lambda_1 + \frac{1}{2} \frac{\partial^2(pb_1)}{\partial x_1^2}  \lambda_1^2\\
T_5 &\approx&  pb_2 -\frac { \partial( pb_2 )}{\partial x_2} \lambda_2 + \frac{1}{2} \frac{\partial^2(pb_2)}{\partial x_2^2}  \lambda_2^2\\
T_6 &\approx&  pd_2 -\frac { \partial( pd_2 )}{\partial x_2} \lambda_2 + \frac{1}{2} \frac{\partial^2(pd_2)}{\partial x_2^2}  \lambda_2^2\\
T_7 &\approx&  pm_{12} +\frac { \partial( pm_{12} )}{\partial x_1} \lambda_1 - \frac { \partial( pm_{12} )}{\partial x_2} \lambda_2+ \frac{1}{2}  \sum_{i=1}^2 \sum_{j=1}^2 (-1)^{i+j} \frac{\partial^2(pm_{12})}{\partial x_i \partial x_j} \lambda_i \lambda_j\\
T_8 &\approx&  pm_{21} -\frac { \partial( pm_{21} )}{\partial x_1} \lambda_1 + \frac { \partial( pm_{21} )}{\partial x_2} \lambda_2+ \frac{1}{2}  \sum_{i=1}^2 \sum_{j=1}^2 (-1)^{i+j} \frac{\partial^2(pm_{21})}{\partial x_i \partial x_j} \lambda_i \lambda_j\\
T_9 &\approx&  pm_{11} +\frac { \partial( pm_{11} )}{\partial x_1} \lambda_1 + \frac { \partial( pm_{11} )}{\partial x_2} \lambda_2+ \frac{1}{2}  \sum_{i=1}^2 \sum_{j=1}^2 (-1)^{i+j} \frac{\partial^2(pm_{11})}{\partial x_i \partial x_j} \lambda_i \lambda_j\\
T_{10} &\approx&  pm_{22} -\frac { \partial( pm_{22} )}{\partial x_1} \lambda_1 - \frac { \partial( pm_{22} )}{\partial x_2} \lambda_2+ \frac{1}{2}  \sum_{i=1}^2 \sum_{j=1}^2 (-1)^{i+j} \frac{\partial^2(pm_{22})}{\partial x_i \partial x_j} \lambda_i \lambda_j
\end{eqnarray*}
Substituting these expressions into (\ref{Equation5.1Allen}) and assuming that $\Delta t$, $\lambda_1 $ and $\lambda_2$ are small, then it is seen that $p(t, x_1 , x_2 )$ approximately solves the Fokker-Planck equation
\begin{eqnarray}\label{Equation5.2Allen}
\frac{\partial p(t, x_1 , x_2 )}{\partial t}&=& - \sum_{i=1}^2 \frac{\partial}{\partial x_1}\left [\mu_i(t, x_1 , x_2 ) p(t, x_1 , x_2 )\right]\nonumber \\ 
&&+  \frac{1}{2}\sum_{i=1}^2 \sum_{j=1}^2 \frac{\partial}{\partial x_i \partial x_j}\left[\sum_{k=1}^2 b_{ik}(t,x_1,x_2)  b_{jk}(t,x_1,x_2) p(t, x_1 , x_2 ) \right] 
\end{eqnarray} 
where $\mu=(\mu_1,\mu_2)$ and $B=\{b_{ij}\}_{1\leq i,j\leq 2}$. On the other hand, it is well known that the probability distribution $p(t, x_1, x_2)$ that solves equation (\ref{Equation5.2Allen}) coincides with the distribution of the solution at time $t$ to the following system of stochastic differential equations
\begin{equation}\label{SDEAllen}
dS = \mu(t, S) dt + B(t, S) dW(t),\quad S(0)=S_0
\end{equation}
where $W$ is a two-dimensional standard Brownian motion and $S_0$ is a given deterministic initial condition. The stochastic differential equation (\ref{SDEAllen}) describes the random evolution of the two-state system $S$ related to the changes described in Table \ref{tab:table1}. 

\subsection{The \emph{Greenhalgh et al.} \cite{Mao paper} model}\label{section Mao}

We now specialize the general model introduced in the previous section to the case investigated in Greenhalgh et al. \cite{Mao paper} (where the process $(S_1,S_2)$ is denoted as $(S,I)$). The values of the parameters in Table \ref{tab:table1} are chosen as follows: 

\begin{table}[h!]
  \centering
     \caption{Probabilities in \emph{Greenhalgh et al.}'s paper}
     \label{tab:table2}
     \begin{tabular}{lcl}
       \toprule Change & & Probability\\
       \midrule 

$\Delta \mathbf{S}^{(1)} = [-1,0]^{T} $&  & $ \mu S_1\Delta t $\\ 

$ \Delta \mathbf{S}^{(2)} = [1,0]^{T}       $ & & $ \mu N\Delta t $\\

$ \Delta \mathbf{S}^{(3)} = [0,-1]^{T} $ &  & $ \mu S_2\Delta t $\\

$ \Delta \mathbf{S}^{(4)} = [0,1]^{T} $ &  & $ 0 $\\
       
$ \Delta \mathbf{S}^{(5)} = [-1,1]^{T} $ & & $ \frac{\lambda(N)S_1S_2}{N}\Delta t $\\
       
$ \Delta \mathbf{S}^{(6)} = [1,-1]^{T} $&  & $ \gamma S_2\Delta t $\\

$ \Delta \mathbf{S}^{(7)} = [-1,-1]^{T} $ &  & $ 0 $\\

$ \Delta \mathbf{S}^{(8)} = [1,1]^{T} $ &  & $ 0 $\\

$ \Delta \mathbf{S}^{(9)} = [0,0]^{T} $ &  & $1 - \sum_{j=1}^{8}p_{j} $\\
      
       \bottomrule
     \end{tabular}
   \end{table}

\noindent where $N:=S_1+S_2$, $\lambda:[0,+\infty[\to [0,+\infty[$ is a continuous monotone increasing function and $\mu$ and $\gamma$ are positive constants. We refer to the paper \cite{Mao paper} for the biological interpretation of these quantities. Now, according to Table \ref{tab:table2} the vector $\mu$ and matrix $V$ in (\ref{mean and variance}) read
\begin{eqnarray*}
    \mu(t,S_1,S_2) = \begin{bmatrix} -\frac{\lambda(N)S_1S_2}{N}+(\mu+\gamma)S_2  \\ \frac{\lambda(N)S_1S_2}{N}-(\mu+\gamma)S_2 
   \end{bmatrix}
   \end{eqnarray*}
 and  
    \begin{eqnarray*}
     \begin{split} V(t,S_1,S_2) =\begin{bmatrix} a & b   \\ b & c
   \end{bmatrix} 
   \end{split}
   \end{eqnarray*}
where to ease the notation we set 
\begin{eqnarray*}
a&:=&\frac{\lambda(N)S_1S_2}{N}+(\mu+\gamma)S_2+2\mu S_1\\
b&:=&-\frac{\lambda(N)S_1S_2}{N}-\gamma S_2\\
c&:=&\frac{\lambda(N)S_1S_2}{N}+(\mu+\gamma)S_2. 
\end{eqnarray*}
Therefore,
   \begin{eqnarray*}
     \begin{split} B(t,S_1,S_2)=V(t,S_1,S_2)^{\frac{1}{2}}  =\frac{1}{d}\begin{bmatrix} a+w & b   \\ b & c+w
   \end{bmatrix}
   \end{split}
   \end{eqnarray*}
with 
\begin{eqnarray*}
w:=\sqrt{ac-b^2}\quad\mbox{ and }\quad d:=\sqrt{a+c+2w}. 
\end{eqnarray*}
We are then lead to study the following two dimensional system of stochastic differential equations
\begin{eqnarray}
\left\{ \begin{array}{ll}
dS_1&=\left[  -\frac{\lambda(N)S_1S_2}{N}+(\mu+\gamma)S_2 \right]dt+\frac{a+w}{d}dW_1+\frac{b}{d}dW_2\label{Mao 1}\\
dS_2&=\left[  \frac{\lambda(N)S_1S_2}{N}-(\mu+\gamma)S_2 \right]dt+\frac{b}{d}dW_1+\frac{c+w}{d}dW_2
\end{array}\right.
\end{eqnarray}
where $W=(W_1,W_2)$ is a standard two dimensional Brownian motion. We observe that by construction
\begin{eqnarray*}
\left( \frac{a+w}{d}\right)^2+\left( \frac{b}{d} \right)^2=a.
\end{eqnarray*}
Therefore, by the martingale representation theorem (see for instance Theorem 3.9 Chapter V in \cite{Revuz Yor}) there exists a Brownian motion $W_3$ such that the first equation in (\ref{Mao 1}) can be rewritten as
\begin{eqnarray*}\label{Mao 1 bis}
dS_1=\left[  -\frac{\lambda(N)S_1S_2}{N}+(\mu+\gamma)S_2 \right]dt+\sqrt{\frac{\lambda(N)S_1S_2}{N}+(\mu+\gamma)S_2+2\mu S_1}dW_3
\end{eqnarray*}
Similarly, since  
\begin{eqnarray*}
\left( \frac{b}{d} \right)^2+\left( \frac{c+w}{d}\right)^2=c
\end{eqnarray*}
by the martingale representation theorem there exists a Brownian motion $W_4$ such that the second equation in (\ref{Mao 1}) can be rewritten as
\begin{eqnarray*}\label{Mao 2 bis}
dS_2=\left[  \frac{\lambda(N)S_1S_2}{N}-(\mu+\gamma)S_2 \right]dt+\sqrt{\frac{\lambda(N)S_1S_2}{N}+(\mu+\gamma)S_2}dW_4.
\end{eqnarray*}
This implies that the system (\ref{Mao 1}) is equivalent to
\begin{eqnarray}
\left\{ \begin{array}{ll}
dS_1&=\left[  -\frac{\lambda(N)S_1S_2}{N}+(\mu+\gamma)S_2 \right]dt+\sqrt{\frac{\lambda(N)S_1S_2}{N}+(\mu+\gamma)S_2+2\mu S_1}dW_3 \label{Mao 2} \\
dS_2&=\left[  \frac{\lambda(N)S_1S_2}{N}-(\mu+\gamma)S_2 \right]dt+\sqrt{\frac{\lambda(N)S_1S_2}{N}+(\mu+\gamma)S_2}dW_4.
\end{array}\right.
\end{eqnarray}
We remark that by construction the Brownian motions $W_3$ and $W_4$ are now correlated. Moreover, if we notice that the drift of the first equation in (\ref{Mao 1}) is the opposite of the one in the second equation in (\ref{Mao 1}), recalling that $N=S_1+S_2$ we may write
\begin{eqnarray*}
dN=\frac{a+b+w}{d}dW_1+\frac{b+c+w}{d}dW_2
\end{eqnarray*}
and, exploiting the definitions of $a$, $b$, $c$, $d$ and $w$, we conclude as before that there exists a Brownian motion $W_5$ such that
\begin{eqnarray}\label{SDE for N}
dN=\sqrt{2\mu N}dW_5.
\end{eqnarray}
Hence, instead of studying the system (\ref{Mao 1}), the authors in \cite{Mao paper} study the equivalent system
\begin{eqnarray}
\left\{ \begin{array}{ll}\label{Mao triangular SDE}
dS_2&=\left[  \frac{\lambda(N)}{N}(N-S_2)S_2-(\mu+\gamma)S_2 \right]dt+\sqrt{\frac{\lambda(N)}{N}(N-S_2)S_2+(\mu+\gamma)S_2}dW_4\\
dN&=\sqrt{2\mu N}dW_5
\end{array}\right.
\end{eqnarray}
where the Brownian motions $W_4$ and $W_5$ are correlated. In the system (\ref{Mao triangular SDE}) the equation for $N$ does not depend on $S_2$ and it belongs to the family of the square root processes (\cite{Mao book}). Once the equation for $N$ is solved, the equation for $S_2$ contains random (for the presence of $N$) H\"older continuous coefficients. Moreover, due to the presence of the square root in the diffusion coefficient of $S_2$, the authors of \cite{Mao paper} consider a modified version of the first equation in (\ref{Mao triangular SDE}) to make the coefficients defined on the whole real line. They consider
\begin{eqnarray}\label{Mao final}
dS_2(t)=\bar{a}(t,N(t),S_2(t))dt+\bar{g}(t,N(t),S_2(t))dW_4(t)
\end{eqnarray} 
where
\begin{eqnarray*}
\bar{a}(t,y,x)=\left\{ \begin{array}{ll}
 0& \mbox{ for }x<0\\
\frac{\lambda(y)x}{y}(y-x)-(\mu+\gamma)x& \mbox{ for }0\leq x\leq y\left(1+\frac{\mu+\gamma}{\lambda(y)}\right)\\
\bar{a}\left(t,y,y\left(1+\frac{\mu+\gamma}{\lambda(y)}\right)\right)& \mbox{ for }x> y\left(1+\frac{\mu+\gamma}{\lambda(y)}\right)
\end{array}\right.
\end{eqnarray*}
and
\begin{eqnarray*}
\bar{g}(t,y,x)=\left\{ \begin{array}{ll}
 0& \mbox{ for }x<0\\
\sqrt{\frac{\lambda(y)x}{y}(y-x)+(\mu+\gamma)x}& \mbox{ for }0\leq x\leq y\left(1+\frac{\mu+\gamma}{\lambda(y)}\right)\\
0& \mbox{ for }x> y\left(1+\frac{\mu+\gamma}{\lambda(y)}\right)
\end{array}\right.
\end{eqnarray*}
The existence of a unique non explosive strong solution to equation (\ref{Mao final}) is obtained through a localization argument in terms of stopping times and comparison inequalities to control the non explosivity of the solution. In the next section we will consider a class of stochastic differential equations, which includes equation (\ref{Mao final}),  allowing for more general models where the existence of a unique non explosive strong solution is proved via a standard Caychy-Euler-Peano approximation method.  

\section{Main theorem}

Motivated by the discussion in the previous sections, we are now ready to state and prove the main result of our manuscript. We begin by specifying the class of coefficients involved in the stochastic differential equations under investigation.\\  
\noindent Let $g:[0,+\infty[\times\mathbb{R}\times\mathbb{R}\to\mathbb{R}$ be a function of the form
\begin{eqnarray}\label{def g}
g(t,y,x)=\sqrt{-x^2+\alpha(t,y)x+\beta(t,y)}
\end{eqnarray}
where $\alpha,\beta:[0,+\infty[\times\mathbb{R}\to\mathbb{R}$ are measurable functions satisfying the condition
\begin{eqnarray}\label{two roots condition}
\alpha(t,y)^2+4\beta(t,y)\geq0\quad\mbox{ for all }(t,y)\in[0,+\infty[\times\mathbb{R}.
\end{eqnarray}
We observe that condition (\ref{two roots condition}) implies that 
\begin{eqnarray*}
-x^2+\alpha(t,y)x+\beta(t,y)\geq 0\quad\mbox{ if and only if }\quad r_1(t,y)\leq x\leq r_2(t,y)
\end{eqnarray*}
where we set
\begin{eqnarray*}
r_1(t,y):=\frac{\alpha(t,y)-\sqrt{\alpha(t,y)^2+4\beta(t,y)}}{2}
\end{eqnarray*}
and
\begin{eqnarray*}
 r_2(t,y):=\frac{\alpha(t,y)+\sqrt{\alpha(t,y)^2+4\beta(t,y)}}{2}.
\end{eqnarray*}
Now, we define
\begin{eqnarray}\label{def bar g}
\bar{g}(t,y,x):=\left\{ \begin{array}{ll}\label{extended sigma}
0 &\mbox{if }\quad x<r_1(t,y) \\
g(t,y,x) &\mbox{if }\quad r_1(t,y)\leq x\leq r_2(t,y) \\
0 &\mbox{if }\quad x>r_2(t,y)
\end{array}\right.
\end{eqnarray}
The function $\bar{g}$ will be the diffusion coefficient of our stochastic differential equation. 
\begin{assumption}\label{assumption on g}
There exist a positive constant $M$ such that
\begin{eqnarray}\label{assumption on alpha}
|\alpha(t,y)|\leq M(1+|y|)\quad\mbox{ and }\quad |\beta(t,y)|\leq M(1+|y|)
\end{eqnarray}
for all $(t,y)\in [0,\infty[\times\mathbb{R}$. Moreover, there exists a positive constant $H$ such that
\begin{eqnarray}\label{holder g}
|\bar{g}(t,y_1,x_1)-\bar{g}(t,y_2,x_2)|\leq H(\sqrt{|y_1-y_2|}+\sqrt{|x_1-x_2|})
\end{eqnarray}
for all $t\in [0,\infty[$ and $y_1,y_2,x_1,x_2\in\mathbb{R}$.
\end{assumption}
We observe that assumption (\ref{assumption on alpha}) implies the bound 
\begin{eqnarray*}
|\bar{g}(t,y,x)|&\leq&\max_{x\in\mathbb{R}}|\bar{g}(t,y,x)|\\
&=&\sqrt{\frac{\alpha(t,y)^2}{4}+\beta(t,y)}\\
&\leq& M(1+|y|)
\end{eqnarray*}
for all $t\in [0,\infty[$ and $y\in\mathbb{R}$. Here the constant $M$ may differ from the one appearing in (\ref{assumption on alpha}); we will adopt this convention for the rest of the paper. We also remark that by construction inequality (\ref{holder g}) for $y_1=y_2$ is  satisfied with a constant $H=\sqrt{|\alpha(t,y_1)|}$.\\

\noindent We now introduce the drift coefficient of our SDE. We start with a measurable function $a:[0,+\infty[\times\mathbb{R}\times\mathbb{R}\to\mathbb{R}$ with the following property.
\begin{assumption}\label{assumption on a}
There exists a positive constant $M$ such that
\begin{eqnarray}\label{assumption growth a}
|a(t,y,x)|\leq M(1+|y|+|x|)
\end{eqnarray}
for all $t\in [0,\infty[$ and $x,y\in\mathbb{R}$. Moreover, there exists a positive constant $L$ such that
\begin{eqnarray}\label{lipschitz a}
|a(t,y_1,x_1)-a(t,y_2,x_2)|\leq L(|y_1-y_2|+|x_1-x_2|)
\end{eqnarray}
for all $t\in [0,\infty[$ and $y_1,y_2,x_1,x_2\in\mathbb{R}$.
\end{assumption}
Then, we set
\begin{eqnarray}\label{def bar a}
\bar{a}(t,y,x):=\left\{ \begin{array}{ll}\label{extended functions}
a(t,y,r_1(t,y)) &\mbox{if }\quad x<r_1(t,y) \\
a(t,y,x) &\mbox{if }\quad r_1(t,y)\leq x\leq r_2(t,y) \\
a(t,y,r_2(t,y)) &\mbox{if }\quad x>r_2(t,y)
\end{array}\right.
\end{eqnarray}
Observe that by construction also the function $\bar{a}$ satisfies Assumption \ref{assumption on a}.\\

\noindent We now consider the following one dimensional stochastic differential equation
\begin{eqnarray}\label{SDE}
dX_t=\bar{a}(t,Y_t,X_t)dt+\bar{g}(t,Y_t,X_t)dW^2_t,\quad X_0=x\in\mathbb{R}
\end{eqnarray}
where $\{Y_t\}_{t\geq 0}$ is the unique strong solution of the stochastic differential equation
\begin{eqnarray}\label{SDE Y}
dY_t=m(t,Y_t)dt+\sigma(t,Y_t)dW^1_t,\quad Y_0=y\in\mathbb{R}.
\end{eqnarray} 
Here $\{(W^1_t,W_t^2)\}_{t\geq 0}$ is a two dimensional correlated Brownian motion defined on a complete filtered probability space $(\Omega,\mathcal{F},\mathbb{P}, \{\mathcal{F}_t\}_{t\geq 0})$ where the filtration $\{\mathcal{F}_t\}_{t\geq 0}$ is generated by the process $\{(W_t^1,W_t^2)\}_{t\geq 0}$. Strong solutions are meant to be $\{\mathcal{F}_t\}_{t\geq 0}$-adapted. \\
Regarding equation (\ref{SDE Y}), the coefficients $m$ and $\sigma$ are assumed to entail existence and uniqueness of a strong solution $\{Y_t\}_{t\geq 0}$ such that
\begin{eqnarray*}
E\Big[\sup_{t\in [0,T]}|Y_t|^2\Big]\mbox{ is finite for all }T>0.
\end{eqnarray*}  
Equations (\ref{SDE}) and (\ref{SDE Y}) describe a class of equations which includes equations (\ref{Mao final}) and (\ref{SDE for N}) as a particular case. 

\begin{remark}
If $r_1(t,y)=r_2(t,y)$ for all $(t,y)\in [0,\infty[\times\mathbb{R}$, which is equivalent to say that $\alpha(t,y)^2+4\beta(t,y)=0$, then the diffusion coefficient $\bar{g}$ is identically zero and the drift coefficient becomes $\bar{a}(t,y,x)=a(t,y,\alpha(t,y)/2)$. Therefore, in this particular case the SDE (\ref{SDE}) takes the form
\begin{eqnarray*}
dX_t=a(t,Y_t,\alpha(t,Y_t)/2)dt,\quad X_0=x
\end{eqnarray*}
whose solution is explicitly given by the formula
\begin{eqnarray*}
X_t=x+\int_0^ta(s,Y_s,\alpha(s,Y_s)/2)ds.
\end{eqnarray*}
\end{remark}

\begin{theorem}[Strong existence and uniqueness]\label{main theorem}
Let Assumption \ref{assumption on g} and Assumption \ref{assumption on a} be fulfilled. Then, the stochastic differential equation (\ref{SDE}) possesses a unique strong solution $\{X_t\}_{t\geq 0}$.
\end{theorem}

\begin{proof}
To ease the notation we consider the time-homogeneous case and hence we drop the explicit dependence on $t$ from all the  coefficients. \\
\noindent We fix an arbitrary $T >0 $ and prove existence and uniqueness of a solution for the SDE
\begin{eqnarray}\label{eq:25}
X_t=x+\int_0^t\bar{a}(Y_s,X_s)ds+\int_0^t\bar{g}(Y_s,X_s)dW^2_s,\quad X_0=x.
\end{eqnarray}
on the time interval $t\in [0,T]$. The proof for the existence is rather long and proceeds as
follows: using a Cauchy-Euler-Peano approximate solutions technique we
define, associated to a partition $ \Delta_{n} $ of $ [0,T] $ a
stochastic process $ X^n $. We will, at the beginning, prove a
convergence result for  $ X^n $ in the space $L^{1}([0,T]\times\Omega) $, then we will prove a convergence result
for  $X^n$ in the space $ \mathcal{C}[0,T] $ with the norm of
the uniform convergence and this will eventually yield the result.\\

\noindent\textbf{Existence}: We consider a sequence of partitions $\{\Delta_n\}_{n\geq 1}$ of the interval $ [0,T]$ with $\Delta_{n} \subseteq \Delta_{n+1}$. Each partition $\Delta_n$ will consist of a set of $N_n+1$ points $\{t^n_0, t^n_1,..., t^n_{N_n}\}$ satisfying 
\begin{eqnarray*}
0= t^n_{0} < t^n_{1} < \cdots < t^n_{N_n} = T.
\end{eqnarray*} 
We denote by $\Vert\Delta_{n}\Vert:= \max_{0\leq k \leq N_n-1}|t_{k+1}^n -t_{k}^n|$, the mesh of the partition $\Delta_n$, and assume that $\lim_{n\to\infty}\Vert\Delta_{n}\Vert = 0$. In the sequel, we will write $t_k$ instead of $t_k^n$ when the membership to the partition $\Delta_n$ will be clear from the context.\\ 
For a given partition $\Delta_{n} $ we construct a continuous and $\{\mathcal{F}_t\}_{t\geq 0}$-adapted stochastic process $\{ X^n_t\}_{t\in [0,T]}$ as follows: for $t=0$ we set $X_t^n=x$ while for 
$t \in ]t_k,t_{k+1}]$ we define
\begin{eqnarray}\label{euler} 
X^n_t := X^n_{t_k}+\bar{a}(Y_{t_k},X^n_{t_k})(t - t_k) + g(Y_{t_k},X^n_{t_k})(W_t-W_{t_k}).
\end{eqnarray}
It is useful to observe that, denoting $\eta_{n}(t) = t_k $ when $ t \in ]t_k,
t_{k+1}]$, we may represent $ X^n_t $ in the compact form:
\begin{eqnarray}\label{compact euler}
 X^n_t = x+\int_{0}^{t}\bar{a}(Y_{\eta_{n}(s)},X^n_{\eta_{n}(s)})ds+\int_{0}^{t}\bar{g}(Y_{\eta_{n}(s)}, X^n_{\eta_{n}(s)})dW^2_s.
\end{eqnarray}

\noindent \textbf{Step one:}\emph{ $ \E|X^n_{\eta_{n}(t)}| $ is
uniformly bounded with respect to $ n $ and $ t $}\\

\noindent We begin with equation (\ref{euler}). Using the triangle inequality and upper bounds for $\bar{a}$ and $\bar{g}$ we get
\begin{eqnarray*}\label{eq:30}
\E[|X^n_{t_{k+1}}|] &\leq & \E[|X^n_{t_k}|] + \E [|\bar{a}(Y_{t_k},X^n_{t_k})(t_{k+1}- t_k)|]\\
&& +\E[|\bar{g}(Y_{t_k},X^n_{t_k})(W_{t_{k+1}} - W_{t_k})|]\\
&\leq & \E[|X^n_{t_k}|] + M|t_{k+1}- t_k|\E \left[1+|Y_{t_k}|\right]+M|t_{k+1}- t_k|\E \left[|X^n_{t_k}|\right]\\
&&+M\E \left[(1+|Y_{t_k}|)|W_{t_{k+1}} - W_{t_k}|\right]\\
&\leq & (1+M\Vert\Delta_n\Vert)\E[|X^n_{t_k}|] + M|t_{k+1}- t_k|\E\left[1+|Y_{t_k}|\right]\\
&&+\frac{M}{2}\left(\E\left[(1+|Y_{t_k}|)^2\right]+\E\left[|W_{t_{k+1}} - W_{t_k}|^{2}\right]\right)\\
&\leq & (1+M\Vert\Delta_n\Vert)\E[|X^n_{t_k}|] + M\Vert\Delta_n\Vert\sup_{t\in [0,T]}\E\left[1+|Y_{t}|\right]\\
&&+\frac{M}{2}\sup_{t\in [0,T]}\E\left[(1+|Y_{t}|)^2\right]+\frac{M}{2}|t_{k-1}-t_k|\\
&\leq & (1+M\Vert\Delta_n\Vert)\E[|X^n_{t_k}|] + M\Vert\Delta_n\Vert\sup_{t\in [0,T]}\E\left[1+|Y_{t}|\right]\\
&&+\frac{M}{2}\sup_{t\in [0,T]}\E\left[(1+|Y_{t}|)^2\right]+\frac{M}{2}\Vert\Delta_n\Vert\\
&\leq&(1+M\Vert\Delta_n\Vert)\E[|X^n_{t_k}|]+\frac{M}{2} \sup_{t\in [0,T]}\E\left[(1+|Y_{t}|)^2\right]+\varepsilon.
\end{eqnarray*}
Here we used the fact that $\Vert\Delta_n\Vert$ tends to zero as $n$ tends to infinity and that $\sup_{t\in [0,T]}\E\left[1+|Y_{t}|\right]$ is finite: we can therefore choose $n$ big enough to make
\begin{eqnarray*}
M\Vert\Delta_n\Vert\sup_{t\in [0,T]}\E\left[1+|Y_{t}|\right]+\frac{M}{2}\Vert\Delta_n\Vert
\end{eqnarray*}
smaller than a given positive $\varepsilon$. Comparing the first and last terms of the previous chain of inequalities we get for all $k\in\{0,...,N_n-1\}$
\begin{eqnarray*}
\E[|X^n_{t_{k+1}}|] \leq(1+M\Vert\Delta_n\Vert)\E[|X^n_{t_k}|]+\frac{M}{2} \sup_{t\in [0,T]}\E\left[(1+|Y_{t}|)^2\right]+\varepsilon
\end{eqnarray*}
which by recursion implies 
\begin{eqnarray*}
\E[|X^n_{t_{k}}|] &\leq& \gamma_1^{k}|x|+\frac{\gamma_1^{k}-1}{\gamma_1-1}\gamma_2\\
&\leq& \gamma_1^{N_n}|x|+\frac{\gamma_1^{N_n}-1}{\gamma_1-1}\gamma_2
\end{eqnarray*}
where for notational convenience we set
\begin{eqnarray*}
\gamma_1:=1+M\Vert\Delta_n\Vert\quad\mbox{ and }\quad\gamma_2:=\frac{M}{2} \sup_{t\in [0,T]}\E\left[(1+|Y_{t_k}|)^2\right]+\varepsilon.
\end{eqnarray*}
Since $ \eta_{n}(t) $ is a step function in $ [0,T]$ with values $\{t_0, t_1,..., t_{N_n}\}$, the previous estimate for $k\in\{0,...,N_n-1\}$ entails the boundedness of the function $[0,T]\ni t\to E[|X^n_{\eta_{n}(t)}|]$.\\
We now obtain an estimate for $E[|X^n_{\eta_{n}(t)}|]$ which is also uniform with respect to $n$. Using the triangle inequality in (\ref{compact euler}) we can write 
\begin{eqnarray}\label{eq:42}
\E[|X^n_{\eta_{n}(t)}| ]&\leq& |x| +\E\left[\left|\int_{0}^{\eta_{n}(t)}\bar{a}(Y_{\eta_{n}(s)},X^n_{\eta_{n}(s)})ds\right|\right]\nonumber\\
&&+\E\left[\left|\int_{0}^{\eta_{n}(t)}\bar{g}(Y_{\eta_{n}(s)},X^n_{\eta_{n}(s)})dW^2_s\right|\right].
\end{eqnarray}
For the first expected value on the right hand side above we employ the assumptions on $\bar{a}$:
\begin{eqnarray*}
\E\left[\left|\int_{0}^{\eta_{n}(t)}\bar{a}(Y_{\eta_{n}(s)},X^n_{\eta_{n}(s)})ds\right|\right]&\leq&\E\left[\int_{0}^{t}|\bar{a}(Y_{\eta_{n}(s)},X^n_{\eta_{n}(s)})|ds\right]\\
&\leq&M\E\left[\int_{0}^{t}(1+|X^n_{\eta_{n}(s)}|+|Y_{\eta_{n}(s)}|)ds\right]\\
&=&M\int_0^t\E[|X^n_{\eta_{n}(s)}|]ds+M\int_{0}^{t}\E\left[1+|Y_{\eta_{n}(s)}|\right]ds\\
&\leq&M\int_0^t\E[|X^n_{\eta_{n}(s)}|]ds+MT\sup_{t\in [0,T]}\E\left[1+|Y_{t}|\right].
\end{eqnarray*}
Using the It\^o isometry and the assumptions on $\bar{g}$ we can treat the second expected value as follows:
\begin{eqnarray*}
\E\left[\left|\int_{0}^{\eta_{n}(t)}\bar{g}(Y_{\eta_{n}(s)},X^n_{\eta_{n}(s)})dW^2_s\right|\right]&\leq&\left(\E\left[\left|\int_{0}^{\eta_{n}(t)}\bar{g}(Y_{\eta_{n}(s)},X^n_{\eta_{n}(s)})dW^2_s\right|^2\right]\right)^{\frac{1}{2}}\\
&\leq&\left(\int_{0}^{t}\E[|\bar{g}(Y_{\eta_{n}(s)},X^n_{\eta_{n}(s)})|^2]ds\right)^{\frac{1}{2}}\\
&\leq& M\left(\int_{0}^{t}\E[(1+|Y_{\eta_{n}(s)}|)^2]ds\right)^{\frac{1}{2}}\\
&\leq&M\sqrt{T\sup_{t\in [0,T]}\E[(1+|Y_t|)^2]}.
\end{eqnarray*}
Plugging the last two estimates in (\ref{eq:42}) gives
\begin{eqnarray*}
\E[|X^n_{\eta_{n}(t)}| ]&\leq& |x| +\E\left[\left|\int_{0}^{\eta_{n}(t)}\bar{a}(Y_{\eta_{n}(s)},X^n_{\eta_{n}(s)})ds\right|\right]\nonumber\\
&&+\E\left[\left|\int_{0}^{\eta_{n}(t)}\bar{g}(Y_{\eta_{n}(s)},X^n_{\eta_{n}(s)})dW^2_s\right|\right]\\
&\leq&|x|+M\int_0^t\E[|X^n_{\eta_{n}(s)}|]ds+MT\sup_{t\in [0,T]}\E\left[1+|Y_{t}|\right]\\
&&+M\sqrt{T\sup_{t\in [0,T]}\E[(1+|Y_t|)^2]}\\
&=&G+M\int_0^t\E[|X^n_{\eta_{n}(s)}|]ds
\end{eqnarray*}
where
\begin{eqnarray*}
G:=|x|+MT\sup_{t\in [0,T]}\E\left[1+|Y_{t}|\right]+M\sqrt{T\sup_{t\in [0,T]}\E[(1+|Y_t|)^2]}.
\end{eqnarray*}
By the Gronwall inequality (we proved before that $t\to\E[|X^n_{\eta_{n}(t)}| ]$ is a non negative, bounded and measurable function) we conclude that
\begin{eqnarray}\label{uniform bound}
\E[|X^n_{\eta_{n}(t)}| ]\leq Ge^{Mt}\leq Ge^{MT}
\end{eqnarray}
which provides the desired uniform bound  (with respect to $n$ and $t$) for $\E[|X^n_{\eta_{n}(t)}| ]$.\\

\noindent\textbf{Step two:}  \emph{ $\E[|X_t^n -X^n_{\eta_{n}(t)}|]$ tends to zero as $n$ tends to infinity, uniformly with respect to $t\in [0,T]$}\\

\noindent We proceed as in step one. Recalling the  identity (\ref{compact euler}) we can write
\begin{eqnarray*}
\E[|X_t^n -X^n_{\eta_{n}(t)}|]&=&\E\left[\left|\int_{\eta_n(t)}^{t}\bar{a}(Y_{\eta_{n}(s)},X^n_{\eta_{n}(s)})ds+\int_{\eta_n(t)}^{t}\bar{g}(Y_{\eta_{n}(s)}, X^n_{\eta_{n}(s)})dW^2_s\right|\right]\\
&\leq&\int_{\eta_n(t)}^{t}\E[|\bar{a}(Y_{\eta_{n}(s)},X^n_{\eta_{n}(s)})|]ds+\E\left[\left|\int_{\eta_n(t)}^{t}\bar{g}(Y_{\eta_{n}(s)}, X^n_{\eta_{n}(s)})dW^2_s\right|\right]\\
&\leq&M\int_{\eta_n(t)}^{t}\E[(1+|X^n_{\eta_{n}(s)}|+|Y_{\eta_{n}(s)}|)]ds\\
&&+\left(\E\left[\left|\int_{\eta_n(t)}^{t}\bar{g}(Y_{\eta_{n}(s)}, X^n_{\eta_{n}(s)})dW^2_s\right|^2\right]\right)^{\frac{1}{2}}\\
&\leq&M(t-\eta_n(t))\left(Ge^{MT}+\sup_{t\in [0,T]}\E[1+|Y_t|]\right)\\
&&+\left(\int_{\eta_n(t)}^{t}\E[|\bar{g}(Y_{\eta_{n}(s)}, X^n_{\eta_{n}(s)})|^2]ds\right)^{\frac{1}{2}}\\
&\leq&M(t-\eta_n(t))\left(Ge^{MT}+\sup_{t\in [0,T]}\E[1+|Y_t|]\right)\\
&&+M\sqrt{t-\eta_n(t)}\sqrt{\sup_{t\in [0,T]}\E[(1+|Y_t|)^2]}\\
&\leq&M\sqrt{\Vert\Delta_n\Vert}\left(Ge^{MT}+\sup_{t\in [0,T]}\E[1+|Y_t|]+\sqrt{\sup_{t\in [0,T]}\E[(1+|Y_t|)^2]}\right).
\end{eqnarray*}
Here, in the third equality, we utilized the uniform upper bound (\ref{uniform bound}). We have therefore proved that
\begin{eqnarray*}
\E[|X_t^n -X^n_{\eta_{n}(t)}|]&\leq&M\sqrt{\Vert\Delta_n\Vert}\left(Ge^{MT}+\sup_{t\in [0,T]}\E[1+|Y_t|]+\sqrt{\sup_{t\in [0,T]}\E[(1+|Y_t|)^2]}\right)\\
&=:&M_1\sqrt{\Vert\Delta_n\Vert}
\end{eqnarray*}
This in turn implies that $\E[|X_t^n -X^n_{\eta_{n}(t)}|]$ tends to zero as $n$ tends to infinity, uniformly with respect to $t\in [0,T]$.\\

\noindent\textbf{Step three:}  \emph{$\{X^n\}_{n \geq 1}$ is a Cauchy
sequence in $L^{1}([0,T]\times \Omega)$}. \\

\noindent We need to prove that for any $\varepsilon>0$ there exists $n_{\varepsilon}\in\mathbb{N}$ such that
\begin{eqnarray*}
\E\left[\int_0^T|X^n_t-X^m_t|dt\right]<\varepsilon\quad\mbox{ for all }n,m\geq n_{\varepsilon}.
\end{eqnarray*}
We have:
\begin{eqnarray*}\label{eq:48} 
X^n_t -X^m_t &=&\int_{0}^{t}\left[\bar{a}(Y_{\eta_n(s)}.X^n_{\eta_{n}(s)}) -\bar{a}(Y_{\eta_m(s)},X^m_{\eta_{m}(s)})\right]ds \\  
&&+\int_{0}^{t}\left[\bar{g}(Y_{\eta_n(s)},X^n_{\eta_{n}(s)}) -
\bar{g}(Y_{\eta_m(s)},X^m_{\eta_{m}(s)})\right]dW^2_s
\end{eqnarray*}
We now aim to apply the It\^o formula to the stochastic process $\{X^n_t-X^m_t\}_{t\in [0,T]}$ for a suitable smooth function that we now describe.\\
Consider the decreasing sequence of real numbers $\{a_h\}_{h\geq 0}$ defined by induction as follows:
\begin{eqnarray*}
a_0=1\quad\mbox{and for $h\geq 1$ }\int^{a_{h}}_{a_{h-1}}\frac{1}{u}du=h.
\end{eqnarray*}
It is easy to see that $\displaystyle a_{h}=e^{-\frac{h(h+1)}{2}} $ and therefore that  $\lim_{h\to +\infty} a_{h}=0$. Define the function 
$\Phi_{h}(u)$ for $u\in [0,\infty) $ such that  $\Phi_{h}(0) =0 $, $\Phi_{h}(u) \in \mathcal{C}^{2}([0,\infty[) $  and
\begin{equation}\label{eq:50} 
\Phi_{h}''(u) =
 \begin{cases} 0, & 0\leq u \leq a_{h}\\
       \mathrm{a~ value~
         between~} 0 \mathrm{~and~} \frac{2}{hu},& a_{h} < u < a_{h-1}
       \\ 0, & u \geq a_{h-1}
     \end{cases}
\end{equation}
with
\begin{eqnarray*}
\int_{a_{h}}^{a_{h-1}}\Phi_{h}''(u)du =1.
\end{eqnarray*}
Integrating $\Phi_{h}'' $ we get
   \begin{equation}
     \label{eq:51} \Phi_{h}'(u) =
     \begin{cases} 0, & 0\leq u \leq a_{h}\\ \mathrm{a~ value~
between~} 0 \mathrm{~and~}1,& a_{h} < u < a_{h-1}\\ 1, &  u \geq
a_{h-1}
     \end{cases}
   \end{equation}
Finally we choose $ \theta_{h}(u) = \Phi_{h}(|u|) $. Then, we have:
\begin{eqnarray*}\label{eq:49} 
\theta_h(X^n_t -X^m_t)&=&\int_{0}^{t}\theta_h'(X^n_s -X^m_s)\left[\bar{a}(Y_{\eta_n(s)}.X^n_{\eta_{n}(s)}) -\bar{a}(Y_{\eta_m(s)},X^m_{\eta_{m}(s)})\right]ds\\ 
&&+\int_{0}^{t}\theta_h'(X^n_s -X^m_s)\left[\bar{g}(Y_{\eta_n(s)},X^n_{\eta_{n}(s)}) -
\bar{g}(Y_{\eta_m(s)},X^m_{\eta_{m}(s)})\right]dW^2_s\\
&&+\frac{1}{2}\int_0^t\theta''(X^n_s -X^m_s)\left[\bar{g}(Y_{\eta_n(s)},X^n_{\eta_{n}(s)}) -
\bar{g}(Y_{\eta_m(s)},X^m_{\eta_{m}(s)})\right]^2ds\\
&=:& I_{1}(\theta_h)+I_{2}(\theta_h) + I_{3}(\theta_h)
\end{eqnarray*}
Since for any $h\geq 0$ and $u\in\mathbb{R}$ we have by construction that $|u|-a_{h-1}\leq \theta_h(u)$, we can write
\begin{eqnarray}\label{L^1 estimate}
\E[|X^n_t-X^m_t|]&\leq&a_{h-1}+\E[\theta_h(X^n_t-X^m_t)]\nonumber\\
&=&a_{h-1}+\E[I_{1}(\theta_h)+I_{2}(\theta_h) + I_{3}(\theta_h)]\nonumber\\
&=&a_{h-1}+\E[I_{1}(\theta_h)]+\E[I_{3}(\theta_h)].
\end{eqnarray} 
Let us now estimate $\E[|I_{1}(\theta_{h})|]$:
\begin{eqnarray*} 
\E[|I_{1}(\theta_{h})|]&=&\E\left[\left| \int_{0}^{t}\theta_h'(X^n_s -X^m_s)\left[\bar{a}(Y_{\eta_n(s)}.X^n_{\eta_{n}(s)}) -\bar{a}(Y_{\eta_m(s)},X^m_{\eta_{m}(s)})\right]ds\right|\right]\\
 &\leq& \E\left[ \int_{0}^{t}|\theta_h'(X^n_s -X^m_s)|\cdot |\bar{a}(Y_{\eta_n(s)}.X^n_{\eta_{n}(s)}) -\bar{a}(Y_{\eta_m(s)},X^m_{\eta_{m}(s)})|ds\right]\\
 &\leq& \E\left[ \int_{0}^{t}|\bar{a}(Y_{\eta_n(s)}.X^n_{\eta_{n}(s)}) -\bar{a}(Y_{\eta_m(s)},X^m_{\eta_{m}(s)})|ds\right]\\
 &\leq&L\int_{0}^{t}\E[|X^n_{\eta_{n}(s)}-X^m_{\eta_{m}(s)}|]ds+L\int_{0}^{t}\E[|Y_{\eta_{n}(s)}-Y_{\eta_{m}(s)}|]ds
 \end{eqnarray*}
In the second inequality we utilized the bound $ |\theta_{h}'(u)| \leq 1 $ which is valid for all $h\geq 0$ and $u\in\mathbb{R}$. By means of the estimate obtained in step two we can write
\begin{eqnarray*}
\E[|X^n_{\eta_{n}(s)}-X^m_{\eta_{m}(s)}|]&\leq&\E[|X^n_{\eta_{n}(s)}-X^n_s|]+\E[|X^n_s-X^m_s|]+\E[|X^m_s-X^m_{\eta_{m}(s)}|]\\
&\leq&M_1(\sqrt{\Vert\Delta_n\Vert}+\sqrt{\Vert\Delta_m\Vert})+\E[|X^n_s-X^m_s|].
\end{eqnarray*}
Similarly we get
\begin{eqnarray*}
\E[|Y_{\eta_{n}(s)}-Y_{\eta_{m}(s)}|]&\leq& \E[|Y_{\eta_{n}(s)}-Y_s|]+\E[Y_s-Y_{\eta_{m}(s)}|]\\
&\leq&C(\sqrt{\Vert\Delta_n\Vert}+\sqrt{\Vert\Delta_m\Vert})
\end{eqnarray*}   
 where the last inequality is due to well known estimates for strong solutions of stochastic differential equations. Combining the last two bounds we conclude that 
 \begin{eqnarray}\label{I_1} 
\E[|I_{1}(\theta_{h})|]&\leq&L\int_{0}^{t}\E[|X^n_{\eta_{n}(s)}-X^m_{\eta_{m}(s)}|]ds+L\int_{0}^{t}\E[|Y_{\eta_{n}(s)}-Y_{\eta_{m}(s)}|]ds\nonumber\\
&\leq&TL(M_1+C)(\sqrt{\Vert\Delta_n\Vert}+\sqrt{\Vert\Delta_n\Vert})+L\int_0^t\E[|X^n_s-X^m_s|]ds.
\end{eqnarray}
We now treat $\E [I_{3}(\theta_{h})]$; by the assumption (\ref{holder g}) and properties of $\theta_h$ we get:
\begin{eqnarray}\label{I_3}
\E [I_{3}(\theta_{h})]&=&\frac{1}{2}\E\left[\int_0^t\theta_h''(X^n_s -X^m_s)(\bar{g}(Y_{\eta_n(s)},X^n_{\eta_{n}(s)})-\bar{g}(Y_{\eta_m(s)},X^m_{\eta_{m}(s)}))^{2}ds\right] \nonumber\\ 
&\leq&\frac{H^2}{2}\E\left[\int_0^t\theta_h''(X^n_s -X^m_s)\left(\sqrt{|X^n_{\eta_{n}(s)}-X^m_{\eta_{m}(s)}|}+\sqrt{|Y_{\eta_{n}(s)}-Y_{\eta_{m}(s)}|}\right)^2ds\right] \nonumber\\
&\leq&H^2\E\left[\int_0^t\theta_h''(X^n_s -X^m_s)\left(|X^n_{\eta_{n}(s)}-X^m_{\eta_{m}(s)}|+|Y_{\eta_{n}(s)}-Y_{\eta_{m}(s)}|\right)ds\right] \nonumber\\
&\leq& H^2\E\left[\int_{0}^{t}\frac{2}{h|X^n_s -X^m_s|}|X^n_s -X^m_s|ds\right]\nonumber\\
&&+H^2\Vert\theta_{h}''\Vert\E\left[\int_{0}^{t}(|X^n_{\eta_{n}(s)} - X^n_s| +|X^m_{\eta_{m}(s)}- X^m_s|)ds\right]\nonumber\\ 
&&+H^2\Vert\theta_{h}''\Vert\E\left[\int_{0}^{t}(|Y_{\eta_{n}(s)} - Y_s| +|Y_s-Y_{\eta_{m}(s)}|)ds\right]\nonumber\\
&\leq& \frac{2H^2T}{h}+\Vert\theta_{h}''\Vert TH^2(M_1+C)(\sqrt{\Vert\Delta_n\Vert}+\sqrt{\Vert\Delta_n\Vert}).
\end{eqnarray}
Here $\Vert\theta_{h}''\Vert$ denotes the supremum norm of $\theta_h''$ while in the last inequality we used the same bound to obtain inequality (\ref{I_1}). Now, let us fix $\varepsilon >0 $. For this $\varepsilon $ let $ h $ be such that $ 0 < a_{h-1}< \varepsilon $ and $ \frac{2H^2T}{h}< \varepsilon$. With this $ h $ being so chosen and fixed, $\Vert\theta''_{h}\Vert$ is bounded. Then, there exists $ n_{\varepsilon} \in \N $ such that
\begin{eqnarray*} 
(M_1+C)(T+\Vert\theta_{h}''\Vert TH^2)(\sqrt{\Vert\Delta_n\Vert}+\sqrt{\Vert\Delta_n\Vert})<\varepsilon
\end{eqnarray*}
 for all $ n,m \geq n_{\varepsilon} $. We can now insert estimates (\ref{I_1}) and (\ref{I_3}) in (\ref{L^1 estimate}) to obtain
 \begin{eqnarray*}
 \E[|X^n_t-X^m_t|]&\leq&a_{h-1}+\E[I_{1}(\theta_h)]+\E[I_{3}(\theta_h)]\\
 &\leq& a_{h-1}+TL(M_1+C)(\sqrt{\Vert\Delta_n\Vert}+\sqrt{\Vert\Delta_n\Vert})+L\int_0^t\E[|X^n_s-X^m_s|]ds\\
 &&+\frac{2H^2T}{h}+\Vert\theta_{h}''\Vert TH^2(M_1+C)(\sqrt{\Vert\Delta_n\Vert}+\sqrt{\Vert\Delta_n\Vert})\\
 &\leq& 3\varepsilon+L\int_0^t\E[|X^n_s-X^m_s|]ds.
 \end{eqnarray*}
 By Gronwall's inequality we conclude then that
 \begin{eqnarray*}
 \E[|X^n_t - X^m_t|] \leq 3 e^{Lt}\varepsilon\leq 3 e^{LT}\varepsilon,
 \end{eqnarray*}
for all $ n,m \geq n_{\varepsilon}$ and all $ t \in [0,T] $. Hence,
\begin{eqnarray*}
\E\left[\int_0^T|X^n_t - X^m_t|dt\right]&=&\int_0^T\E[|X^n_t - X^m_t|]dt\\
&\leq& T\sup_{t\in[0,T]}\E[|X^n_t - X^m_t|]\\
&\leq&3Te^{LT}\varepsilon.
\end{eqnarray*}
The claim of step three is proved.\\

\noindent\textbf{Step four:}  \emph{$\{X^n\}_{n \geq 1}$ is a Cauchy
sequence in $L^{1}(\Omega; C([0,T]))$}. \\

\noindent We know that $ \{X^n\}_{n \geq 1} $ is a Cauchy sequence in $ L^{1}([0,T]\times \Omega) $ which is
a complete space. We can therefore conclude that there exists a stochastic process $X \in L^{1}([0,T]\times \Omega) $ such that 
\begin{eqnarray*}
 \lim_{n \to \infty}\E\left[\int_0^T|X^n_t-X_t|dt\right]=0.
 \end{eqnarray*}
From Step two we can also deduce that
\begin{eqnarray*}
 \lim_{n \to \infty}\E\left[\int_0^T|X^n_{\eta_n(t)}-X_t|dt\right]=0.
 \end{eqnarray*}
Hence, there exists a subsequence (we keep the same indexes though for easy notations) such that
\begin{eqnarray*} 
 \lim_{n \to \infty}X_t^n(\omega)=\lim_{n \to \infty}X_{\eta_n(t)}^n(\omega) = X_t(\omega)\quad   dt\times d\P\mbox{-almost surely}.
 \end{eqnarray*}
Since the process $\{X^n_t\}_{t\in[0,T]}$ is $\{\mathcal{F}_t\}_{t\in [0,T]}$-adapted for any $n\in\mathbb{N}$ and almost sure convergence preserves measurability, we deduce that $\{X_t\}_{t\in[0,T]}$ is also $\{\mathcal{F}_t\}_{t\in [0,T]}$-adapted. To prove the continuity of $\{X_t\}_{t\in[0,T]}$ we need to check the convergence in the uniform topology, i.e. we need to estimate  
$\E\left[\sup_{t\in [0,T]}|X^n_t - X^m_t|\right]$.\\
As before we employ the representation (\ref{compact euler}):
\begin{eqnarray*}
\E\left[\sup_{t\in [0,T]}|X^n_t - X^m_t|\right]&\leq&\E\left[\sup_{t\in [0,T]}\int_{0}^{t}|\bar{a}(Y_{\eta_n(s)},X^n_{\eta_{n}(s)}) -\bar{a}(Y_{\eta_m(s)},X^m_{\eta_{m}(s)})|ds\right]\\
&&+\E\left[\sup_{t\in [0,T]}\left|\int_{0}^{t}(\bar{g}(Y_{\eta_n(s)},X^n_{\eta_{n}(s)}) -\bar{g}(Y_{\eta_m(s)},X^m_{\eta_{m}(s)}))dW^2_s\right|\right]\\ 
&\leq& \int_{0}^{T}\E[|\bar{a}(Y_{\eta_n(s)},X^n_{\eta_{n}(s)}) -\bar{a}(Y_{\eta_m(s)},X^m_{\eta_{m}(s)})|]ds\\
&&+\E\left[\sup_{t\in [0,T]}\left|\int_{0}^{t}(\bar{g}(Y_{\eta_n(s)},X^n_{\eta_{n}(s)}) -\bar{g}(Y_{\eta_m(s)},X^m_{\eta_{m}(s)}))dW^2_s\right|^2\right]^{\frac{1}{2}}\\
&=:&J_1+J_2
\end{eqnarray*}
 To treat $J_1$ we proceed as before; using inequality (\ref{I_1}) we obtain
 \begin{eqnarray}\label{J_1}
 J_1&=& \int_{0}^{T}\E[|\bar{a}(Y_{\eta_n(s)},X^n_{\eta_{n}(s)}) -\bar{a}(Y_{\eta_m(s)},X^m_{\eta_{m}(s)})|]ds\nonumber\\
 &\leq&L \int_{0}^{T}\E[|X^n_{\eta_{n}(s)}-X^m_{\eta_{m}(s)}|]ds+\int_0^T\E[|Y_{\eta_{n}(s)}-Y_{\eta_{m}(s)}|]ds\\ 
 &\leq&TL(M_1+C)(\sqrt{\Vert\Delta_n\Vert}+\sqrt{\Vert\Delta_n\Vert})+L\int_0^T\E[|X^n_s-X^m_s|]ds\nonumber.
 \end{eqnarray}
Since we proved in Step three that $\{X^n\}_{n\geq 1}$ is a Cauchy sequence in $L^1([0,T[\times\Omega)$ and by assumption $\Vert\Delta_n\Vert$ tends to zero as $n$ tends to infinity, we can find $n$ and $m$ big enough to make the last row of the previous chain of inequalities smaller than any positive $\varepsilon$.  \\
We now evaluate $J_2$. Invoking the Doob maximal inequality and It\^o isometry we can write
\begin{eqnarray*} 
 J_2&=&\E\left[\sup_{t\in [0,T]}\left|\int_{0}^{t}(\bar{g}(Y_{\eta_n(s)},X^n_{\eta_{n}(s)}) -\bar{g}(Y_{\eta_m(s)},X^m_{\eta_{m}(s)}))dW^2_s\right|^2\right]^{\frac{1}{2}}\\
&\leq&2\E\left[\left|\int_{0}^{T}(\bar{g}(Y_{\eta_n(s)},X^n_{\eta_{n}(s)}) -\bar{g}(Y_{\eta_m(s)},X^m_{\eta_{m}(s)}))dW^2_s\right|^2\right]^{\frac{1}{2}}\\
&=&2\E\left[\int_{0}^{T}|\bar{g}(Y_{\eta_n(s)},X^n_{\eta_{n}(s)}) -\bar{g}(Y_{\eta_m(s)},X^m_{\eta_{m}(s)})|^2ds\right]^{\frac{1}{2}}\\
&\leq&2H\E\left[\int_{0}^{T}\left(\sqrt{|X^n_{\eta_{n}(s)}-X^m_{\eta_{m}(s)}|}+\sqrt{|Y_{\eta_{n}(s)}-Y_{\eta_{m}(s)}|}\right)^2ds\right]^{\frac{1}{2}}\\
&\leq&2\sqrt{2}H\E\left[\int_{0}^{T}|X^n_{\eta_{n}(s)}-X^m_{\eta_{m}(s)}|+|Y_{\eta_{n}(s)}-Y_{\eta_{m}(s)}|ds\right]^{\frac{1}{2}}\\
&=&2\sqrt{2}H\left(\int_{0}^{T}\E[|X^n_{\eta_{n}(s)}-X^m_{\eta_{m}(s)}|]+\E[|Y_{\eta_{n}(s)}-Y_{\eta_{m}(s)}|]ds\right)^{\frac{1}{2}}.
\end{eqnarray*}
If we now observe that the last member above is equivalent to (\ref{J_1}), we can proceed as before and conclude that for any $\varepsilon>0$ there exists $n_{\varepsilon}\in\mathbb{N}$ such that
\begin{eqnarray*}
\E\left[\sup_{t\in [0,T]}|X^n_t-X^m_t|\right]<\varepsilon\quad\mbox{ for all }n,m\geq n_{\varepsilon}.
\end{eqnarray*}
This proves that $\{X_n\}_{n\geq 1}$ is a Cauchy sequence in $L^1(\Omega;C([0,T])$ and thus 
\begin{eqnarray*}
\lim_{n\to\infty}\E\left[\sup_{t\in [0,T]}|X^n_t-X_t|\right]=0
\end{eqnarray*}
where $\{X_t\}_{t\in [0,T]}$ is the stochastic process obtained in Step three. Moreover, we can find a subsequence  (we keep the same indexes though for easy notations) such that
\begin{eqnarray*}
\lim_{n\to\infty}\sup_{t\in [0,T]}|X^n_t(\omega)-X_t(\omega)|=0\quad  d\P\mbox{-almost surely}.
\end{eqnarray*}
Since the processes $\{X^n_t\}_{t\in [0,T]}$ are continuous by construction for each $n\in\mathbb{N}$, we deduce that the process $\{X_t\}_{t\in [0,T]}$ is also continuous being a uniform limit of continuous functions.\\

\noindent\textbf{Step five:}  \emph{The stochastic process $\{X_t\}_{t\in [0,T]}$ solves equation (\ref{SDE})}. \\

\noindent Finally we show that   
\begin{eqnarray*} 
 \P\left(X(t) = x + \int_{0}^{t}\bar{a}(Y_s,X_s)ds+\int_{0}^{t}\bar{g}(Y_s,X_s)dW^2_s\quad\mbox{ for all }t\in [0,T]\right)=1. 
 \end{eqnarray*}
 This in turn will be proven by showing that 
\begin{eqnarray*}
\E\left[\sup_{t\in [0,T]}\left|X_t- x-\int_{0}^{t}\bar{a}(Y_s,X_s)ds-\int_{0}^{t}\bar{g}(Y_s,X_s)dW^2_s\right|\right] =0
\end{eqnarray*}
 In fact, the equality
 \begin{eqnarray*}
 &&X_t- x-\int_{0}^{t}\bar{a}(Y_s,X_s)ds-\int_{0}^{t}\bar{g}(Y_s,X_s)dW^2_s\\
 &=&X_t- X^n_{\eta_n(t)}+\int_{0}^{t}\bar{a}(Y_{\eta_n(s)},X^n_{\eta_n(s)})-\bar{a}(Y_s,X_s)ds\\
 &&+\int_{0}^{t}\bar{g}(Y_{\eta_n(s)},X^n_{\eta_n(s)})-\bar{g}(Y_s,X_s)dW^2_s
 \end{eqnarray*}
 implies
 \begin{eqnarray*}
 &&\sup_{t\in [0,T]}\left|X_t- x-\int_{0}^{t}\bar{a}(Y_s,X_s)ds-\int_{0}^{t}\bar{g}(Y_s,X_s)dW^2_s\right|\\
 &\leq& \sup_{t\in [0,T]}|X_t- X^n_{\eta_n(t)}|+\int_{0}^{T}|\bar{a}(Y_{\eta_n(s)},X^n_{\eta_n(s)})-\bar{a}(Y_s,X_s)|ds\\
 &&+ \sup_{t\in [0,T]}\left|\int_{0}^{t}\bar{g}(Y_{\eta_n(s)},X^n_{\eta_n(s)})-\bar{g}(Y_s,X_s)dW^2_s\right|.
 \end{eqnarray*}
 If we take the expectation and use the technique utilized in Step four to bound the terms in the right hand side of the previous inequality we get
 \begin{eqnarray*}
&&\E\left[\sup_{t\in [0,T]}\left|X_t- x-\int_{0}^{t}\bar{a}(Y_s,X_s)ds-\int_{0}^{t}\bar{g}(Y_s,X_s)dW^2_s\right|\right]\\
&=&\lim_{n\to\infty}\E\left[\sup_{t\in [0,T]}\left|X_t- x-\int_{0}^{t}\bar{a}(Y_s,X_s)ds-\int_{0}^{t}\bar{g}(Y_s,X_s)dW^2_s\right|\right]\\
&\leq&\lim_{n\to\infty}\left(\E\left[\sup_{t\in [0,T]}|X_t- X^n_{\eta_n(t)}|\right]+\E\left[\int_{0}^{T}|\bar{a}(Y_{\eta_n(s)},X^n_{\eta_n(s)})-\bar{a}(Y_s,X_s)|ds\right]\right)\\
&&+\lim_{n\to\infty}\E\left[\sup_{t\in [0,T]}\left|\int_{0}^{t}\bar{g}(Y_{\eta_n(s)},X^n_{\eta_n(s)})-\bar{g}(Y_s,X_s)dW^2_s\right|\right]\\
&=&0.
\end{eqnarray*}
 
\noindent\textbf{Uniqueness}:  We use a standard approach. Let $\{X_t\}_{t\in [0,T]}$ and $\{Z_t\}_{t\in [0,T]}$  be two strong solutions of equation (\ref{SDE}). Setting,
 \begin{eqnarray}\label{eq:62}
 \delta_{t}:= X_t -Z_t = \int_{0}^{t}[\bar{a}(Y_s,X_s) - \bar{a}(Y_s,Z_s) ]ds + \int_{0}^{t}[\bar{g}(Y_s,X_s) - \bar{g}(Y_s,Z_s) ]dW^2_s
 \end{eqnarray}
we get by the It\^o formula
\begin{eqnarray*}\label{eq:63}
 \theta_h(\delta_{t})& =& \int_{0}^{t}\theta_h'(\delta_{s})[\bar{a}(Y_s,X_s) - \bar{a}(Y_s,Z_s) ]ds\\
 &&+\int_{0}^{t}\theta_h'(\delta_{s})[\bar{g}(Y_s,X_s) - \bar{g}(Y_s,Z_s) ]dW^2_s\\
 &&+\frac{1}{2}\int_{0}^{t}\theta_h''(\delta_{s})[\bar{g}(Y_s,X_s) - \bar{g}(Y_s,Z_s) ]^2ds
\end{eqnarray*}
where $ \{\theta_h\}_{h\geq 0}$ is the collection of functions defined in Step three. Using the assumptions on $\bar{a}$ and $\bar{g}$ and the bounds $|\theta'_h(u)|\leq 1$ and $|\theta_h''(u)|\leq \frac{2}{hu}$ we get
\begin{eqnarray*}
\E[\theta_{h}(\delta_t)]&\leq& \E\left[\int_{0}^{t}\theta_h'(\delta_{s})[\bar{a}(Y_s,X_s) - \bar{a}(Y_s,Z_s) ]ds\right] + \frac{tH^2}{h} \\
&\leq&  L\int_{0}^{t}\E[|\delta_{s}|]ds + \frac{tH^2}{h}
\end{eqnarray*}
If we let $h \to\infty $, the function $\theta_h$ approaches the absolute value function; hence,  Gronwall's inequality and sample path continuity imply that $\{X_t\}_{t\in [0,T]}$ and $\{Z_t\}_{t\in [0,T]}$ are indistinguishable.
\end{proof}

\end{document}